\documentclass[reqno,12pt]{amsart}
\usepackage{amsfonts}
\usepackage{bbm}
\usepackage{}
\numberwithin{equation}{section}
\usepackage{color}

\usepackage{indentfirst}
\usepackage{color}
\usepackage{amssymb}
\usepackage{mathrsfs}
\usepackage{xy}
\xyoption{all}
\def\Ext{\mbox{\rm Ext}\,} \def\Hom{\mbox{\rm Hom}} \def\dim{\mbox{\rm dim}} 
\def\lr#1{\langle #1\rangle}

\def\zsum{\sum\limits_{i\in\mathbb{Z}_m}}

\def\Aut{\mbox{\rm Aut}}\def\A{\mathcal{A}\,}

\theoremstyle{plain}
\newtheorem{theorem}{\bf Theorem}[section]
\newtheorem{lemma}[theorem]{\bf Lemma}
\newtheorem{corollary}[theorem]{\bf Corollary}

\theoremstyle{definition}
\newtheorem{definition}[theorem]{\bf Definition}
\newtheorem{remark}[theorem]{\bf Remark}
\newtheorem{example}[theorem]{\bf Example}

\newcommand{\bt}{\begin{theorem}}
\newcommand{\et}{\end{theorem}}
\newcommand{\bl}{\begin{lemma}}
\newcommand{\el}{\end{lemma}}
\newcommand{\bd}{\begin{definition}}
\newcommand{\ed}{\end{definition}}
\newcommand{\bc}{\begin{corollary}}
\newcommand{\ec}{\end{corollary}}
\newcommand{\bp}{\begin{proof}}
\newcommand{\ep}{\end{proof}}
\newcommand{\bx}{\begin{example}}
\newcommand{\ex}{\end{example}}
\newcommand{\br}{\begin{remark}}
\newcommand{\er}{\end{remark}}
\newcommand{\be}{\begin{equation}}
\newcommand{\ee}{\end{equation}}
\newcommand{\ba}{\begin{align}}
\newcommand{\ea}{\end{align}}
\newcommand{\bn}{\begin{enumerate}}
\newcommand{\en}{\end{enumerate}}
\newcommand{\bcs}{\begin{cases}}
\newcommand{\ecs}{\end{cases}}

\makeatletter
\renewcommand{\section}{\@startsection{section}{1}{0mm}
  {-\baselineskip}{0.5\baselineskip}{\bf\leftline}}
\makeatother

\begin{document}

\title[A Note on Odd Periodic derived Hall algebras]{A Note on Odd Periodic derived Hall algebras}
\author{Haicheng Zhang$^\ast$, Xinran Zhang, Zhiwei Zhu}
\address{Institute of Mathematics, School of Mathematical Sciences, Nanjing Normal University,
 Nanjing 210023, P. R. China.\endgraf}
\email{zhanghc@njnu.edu.cn (H.Zhang)}
\email{505853930@qq.com (X.Zhang)}
\email{1985933219@qq.com (Z.Zhu)}

\thanks{$\ast$: Corresponding author.}

\begin{abstract}
Let $m$ be an odd positive integer and $D_m(\mathcal {A})$ be the $m$-periodic derived category of a finitary hereditary abelian category $\A$. In this note, we prove that there is an embedding of algebras from the derived Hall algebra of $D_m(\mathcal {A})$ defined by Xu-Chen \cite{XuChen} to the extended derived Hall algebra of $D_m(\mathcal {A})$ defined in \cite{Zhang2}. This homomorphism is given on basis elements, rather than just on generating elements.
\end{abstract}

\maketitle

\section{Introduction}
By Ringel \cite{R90,R90a} and Green \cite{Gr95}, the Ringel--Hall algebra of a hereditary algebra gives a realization of the half part of the corresponding quantum group.
To\"en \cite{Toen} defined the derived Hall algebra for a differential graded category satisfying some finiteness conditions, and
Xiao and Xu \cite{XiaoXu} generalised To\"en's construction to any triangulated category satisfying certain homological finiteness conditions. In 2011, Bridgeland \cite{Br} provided a realization of the entire quantum group by considering the localized Ringel--Hall algebra of 2-periodic complexes of projective modules over a hereditary algebra $A$. Inspired by the work of Bridgeland, for any positive integer $m$, Chen and Deng \cite{ChenD} considered Bridgeland's Hall algebra of $m$-periodic complexes of $A$. In order to generalise Bridgeland's construction to any hereditary abelian categories which may not have enough projectives, one defines the so-called semi-derived Ringel--Hall algebras (cf. \cite{Gorsky,LP,LinP}).

The derived  Hall algebra for an odd periodic triangulated category was defined by Xu-Chen \cite{XuChen}.
Recently, for a hereditary abelian category $\A$, Lin and Peng \cite{LinP} proved that there exists an embedding of algebras from the derived Hall algebra of the odd periodic derived category $D_m(\A)$ to the extended semi-derived Ringel--Hall algebra of $m$-periodic complexes of $\A$.
By analysing the relations between the extensions in the bounded derived category $D^b(\A)$ and the $m$-periodic derived category $D_m(\A)$, Zhang \cite{Zhang2} defined an $m$-periodic extended derived Hall algebra $\mathcal {D}\mathcal {H}_m^{\rm e}(\A)$ for $D_m(\A)$
by applying the (dual) derived Hall numbers of $D^b(\A)$ and proved that it is isomorphic to the twisted semi-derived Ringel--Hall algebra of $m$-periodic complexes of $\A$.
In particular, the $2$-periodic extended derived Hall algebra provides a Hall algebra for the root category of $\A$ and it is proved that $\mathcal {D}\mathcal {H}_2^{\rm e}(\A)$ is isomorphic to the Drinfeld double Hall algebra of $\A$ (cf. \cite{Zhang}).

In this paper, let $m$ be an odd positive integer and $\A$ be a finitary hereditary abelian category. The multiplication structure of the derived Hall algebra $\mathcal {D}\mathcal {H}_m(\A)$ of $D_m(\A)$ defined by Xu-Chen \cite{XuChen} has an explicit characterization in \cite{Zhang2}. Using this characterization and the explicit multiplication structure of $\mathcal {D}\mathcal {H}_m^{\rm e}(\A)$, we establish an embedding of algebras from $\mathcal {D}\mathcal {H}_m(\A)$ to $\mathcal {D}\mathcal {H}_m^{\rm e}(\A)$. Since $\mathcal {D}\mathcal {H}_m^{\rm e}(\A)$ is isomorphic to the twisted semi-derived Ringel--Hall algebra $\mathcal {S}\mathcal {D}\mathcal {H}_{\mathbb{Z}_m}(\A)$ of $m$-periodic complexes of $\A$, we have an embedding of algebras from $\mathcal {D}\mathcal {H}_m(\A)$ to $\mathcal {S}\mathcal {D}\mathcal {H}_{\mathbb{Z}_m}(\A)$. Compared with the embedding given in \cite{LinP}, this embedding is defined on the basis elements of $\mathcal {D}\mathcal {H}_m(\A)$, rather than just on generating elements, and the proof of the algebra homomorphism is direct.

Throughout the paper, $m$ is an odd positive integer, $\mathbb{F}_q$ is a finite field with $q$ elements and $v=\sqrt{q}\in\mathbb{C}$, $\A$ is an essentially small hereditary abelian $\mathbb{F}_q$-category, and $D^b(\mathcal {A})$ is the bounded derived category of $\A$ with the shift functor $[1]$. We always assume that $\A$ is finitary, i.e., for any objects $M,N\in\A$, the spaces $\Hom_{\A}(M,N)$ and $\Ext^1_{\A}(M,N)$ are both finite dimensional. Let $K(\mathcal{A})$ be the Grothendieck group of $\A$, we denote by $\hat{M}$ the image of $M$ in $K(\mathcal{A})$ for any $M\in\A$ and set $\frac{1}{2}K(\mathcal{A}):=\{\frac{1}{2}\alpha~|~\alpha\in K(\mathcal{A})\}$. For a finite set $S$, we denote by $|S|$ its cardinality. For an essentially small finitary category $\mathcal {E}$, we denote by ${\rm Iso}(\mathcal {E})$ the set of isomorphism classes $[X]$ of objects $X$ in $\mathcal {E}$; for each object $X\in\mathcal {E}$, denote by $\Aut_\mathcal {E} (X)$ the automorphism group of $X$. For any object $X\in\A$, we set $a_X:=|\Aut_{\A}(X)|$. We denote the quotient ring $\mathbb{Z}/m\mathbb{Z}$ by $\mathbb{Z}_m=\{0,1,\ldots,m-1\}$, and set $\prod\limits_{i\in\mathbb{Z}_m}a_i:=
a_0 a_1\cdots a_{m-1}$.

\section{Preliminaries}
In this section, we recall the definitions of Euler forms, dual derived Hall numbers and $m$-periodic (extended) derived Hall algebras.
\subsection{Euler forms}
For any objects $M,N \in \mathcal{A}$, set $$\lr{M,N}:=\dim_k\Hom_{\A}(M,N)-\dim_k\Ext^1_{\A}(M,N)$$
and it descends to give a bilinear form
$$\lr{\cdot ,\cdot }: K(\mathcal{A})\times K(\mathcal{A})\longrightarrow \mathbb{Z}$$ known as the \emph{Euler form}. The \emph{symmetric Euler form}
$$(\cdot ,\cdot ): K(\mathcal{A})\times K(\mathcal{A})\longrightarrow \mathbb{Z}$$ is defined by $(\alpha,\beta)=\lr{\alpha,\beta}+\lr{\beta,\alpha}$ for any $\alpha,\beta \in K(\mathcal{A})$.

\subsection{Derived Hall numbers}
For any objects $M,N\in D^b(\A)$, set $$\{M,N\}:=\prod\limits_{i>0}|\Hom_{D^b(\A)}(M[i],N)|^{(-1)^i}.$$
The (dual) \emph{derived Hall algebra} $\mathcal {D}\mathcal {H}(\A)$ of $\A$ is the $\mathbb{C}$-vector space with the basis $\{u_{[X]}~|~[X]\in {\rm Iso}(D^b(\A))\}$, and with the multiplication defined by
$$u_{[X]} u_{[Y]}=\sum\limits_{[L]\in {\rm Iso}(D^b(\A))}H_{X,Y}^Lu_{[L]}$$ where
$$H_{X,Y}^L:=\frac{|\Ext^1_{D^b(\A)}(X,Y)_{L}|}{|\Hom_{D^b(\A)}(X,Y)|}\cdot\frac{1}{\{X,Y\}}$$
and $\Ext^1_{D^b(\A)}(X,Y)_{L}$ denotes the subset of $\Hom_{D^b(\A)}(X,Y[1])$ consisting of morphisms $X\to Y[1]$ whose cone is isomorphic to $L[1]$.
By \cite{Toen,XiaoXu,XiaoXu2}, $\mathcal {D}\mathcal {H}(\A)$ is an associative and unital algebra.

\subsection{$m$-periodic (extended) derived Hall algebras}
By abuse of notation, in what follows, for each object $X\in D_m(\A)$, we may also write $u_X$ for $u_{[X]}$.
\begin{definition}$($\cite[Definition 3.1]{Zhang2}$)$\label{maindef}
The Hall algebra $\mathcal {D}\mathcal {H}_m^{\rm e}(\A)$, called the {\em $m$-periodic extended derived Hall algebra} of $\A$, is the $\mathbb{C}$-vector space with the basis $$\{u_{\bigoplus\limits_{i\in\mathbb{Z}_m}M_i[i]}\prod\limits_{i\in\mathbb{Z}_m}K_{\alpha_i,i}~|~[M_i]\in {\rm Iso}(\A),\alpha_i\in \frac{1}{2}K(\A)~\text{for~all~}i\in\mathbb{Z}_m\}$$ and with the multiplication defined on basis elements by
{\begin{equation*}
\begin{split}&(u_{\bigoplus\limits_{i\in\mathbb{Z}_m}A_i[i]}\prod\limits_{i\in\mathbb{Z}_m}K_{\alpha_i,i})(u_{\bigoplus\limits_{i\in\mathbb{Z}_m}B_i[i]}\prod\limits_{i\in\mathbb{Z}_m}K_{\beta_i,i})=\\&
v^{a_0}\sum\limits_{[I_i],[M_i]\in{\rm Iso}(\A), i\in\mathbb{Z}_m}v^{-(\hat{I}_{m-1},\alpha_0+\beta_0)+\sum\limits_{i=1}^{m-1}(\hat{I}_i,\alpha_{i-1}+\beta_{i-1})+\sum\limits_{i\in\mathbb{Z}_m}\lr{\hat{M}_i-\hat{M}_{i+1},\hat{I}_i}+\sum\limits_{i=1}^{m-1}\lr{\hat{I}_{i-1},\hat{I}_i}-\lr{\hat{I}_0,\hat{I}_{m-1}}}
\\&\quad\quad\quad\quad\quad\quad\quad\quad\quad\quad\quad\prod\limits_{i\in\mathbb{Z}_m}\frac{H_{I_i[1]\oplus A_i,B_i\oplus I_{i-1}[-1]}^{M_i}}{a_{I_i}} u_{\bigoplus\limits_{i\in\mathbb{Z}_m}M_i[i]}\prod\limits_{i\in\mathbb{Z}_m}K_{{\hat{I}}_i+\alpha_i+\beta_i,i}\end{split}\end{equation*}}
where $a_0=\sum\limits_{i\in\mathbb{Z}_m}\lr{\hat{A}_i,\hat{B}_i}+\sum\limits_{i\in\mathbb{Z}_m}(\alpha_i,\hat{B}_i-\hat{B}_{i+1})+\sum\limits_{i=1}^{m-1}(\alpha_i,\beta_{i-1})-(\alpha_{m-1},\beta_0)$.
By convention, $\sum\limits_{i=1}^{m-1}x_i=x_1$, if $m=1$.
In particular, we have that
{\begin{equation}\label{yschengfa}
\begin{split}&u_{\bigoplus\limits_{i\in\mathbb{Z}_m}A_i[i]}u_{\bigoplus\limits_{i\in\mathbb{Z}_m}B_i[i]}=
v^{\sum\limits_{i\in\mathbb{Z}_m}\lr{\hat{A}_i,\hat{B}_i}}\sum\limits_{[I_i],[M_i]\in{\rm Iso}(\A), i\in\mathbb{Z}_m}v^{\sum\limits_{i\in\mathbb{Z}_m}\lr{\hat{M}_i-\hat{M}_{i+1},\hat{I}_i}+\sum\limits_{i=1}^{m-1}\lr{\hat{I}_{i-1},\hat{I}_i}-\lr{\hat{I}_0,\hat{I}_{m-1}}}\\&
\quad\quad\quad\quad\quad\quad\quad\quad\quad\quad\quad\quad\quad\quad\prod\limits_{i\in\mathbb{Z}_m}\frac{H_{I_i[1]\oplus A_i,B_i\oplus I_{i-1}[-1]}^{M_i}}{a_{I_i}} u_{\bigoplus\limits_{i\in\mathbb{Z}_m}M_i[i]}\prod\limits_{i\in\mathbb{Z}_m}K_{{\hat{I}}_i,i},\end{split}\end{equation}}
{\begin{equation}\label{Kjiaohuan}
\begin{split}
&\prod\limits_{i\in\mathbb{Z}_m}K_{\alpha_i,i}\prod\limits_{i\in\mathbb{Z}_m}K_{\beta_i,i}=
v^{-(\alpha_{m-1},\beta_0)+\sum\limits_{i=1}^{m-1}(\alpha_i,\beta_{i-1})}\prod\limits_{i\in\mathbb{Z}_m}K_{\alpha_i+\beta_i,i},\end{split}\end{equation}}
{\begin{equation}\label{kujiaohuan}
\begin{split}&(\prod\limits_{i\in\mathbb{Z}_m}K_{\alpha_i,i})u_{\bigoplus\limits_{i\in\mathbb{Z}_m}B_i[i]}=
v^{\sum\limits_{i\in\mathbb{Z}_m}(\alpha_i,\hat{B}_i-\hat{B}_{i+1})}u_{\bigoplus\limits_{i\in\mathbb{Z}_m}B_i[i]}\prod\limits_{i\in\mathbb{Z}_m}K_{\alpha_i,i}.
\end{split}\end{equation}}
\end{definition}
By (\ref{Kjiaohuan}), it is easy to see that
\begin{equation}\label{kkjiaohuan}
\prod\limits_{i\in\mathbb{Z}_m}K_{\alpha_i,i}\prod\limits_{i\in\mathbb{Z}_m}K_{\beta_i,i}
=v^{\sum\limits_{i\in\mathbb{Z}_m}(\alpha_i,\beta_{i-1}-\beta_{i+1})}\prod\limits_{i\in\mathbb{Z}_m}K_{\beta_i,i}\prod\limits_{i\in\mathbb{Z}_m}K_{\alpha_i,i}.
\end{equation}
Compared with \cite[Definition 3.1]{Zhang2}, the Hall algebra $\mathcal {D}\mathcal {H}_m^{\rm e}(\A)$ here is larger, since the elements $K_{\alpha_i,i}$ here allow $\alpha_i$ to be taken in $\frac{1}{2}K(\A)$, which includes $K(\A)$.
By the same proof of \cite[Theorem 3.3]{Zhang2}, we have the following
\begin{theorem}$($\cite[Theorem 3.3]{Zhang2}$)$
The $m$-periodic extended derived Hall algebra $\mathcal {D}\mathcal {H}_m^{\rm e}(\A)$ is an associative algebra.
\end{theorem}

\begin{definition}$($\cite[Definition 4.1]{Zhang2}$)$\label{jidef}
The Hall algebra $\mathcal {D}\mathcal {H}_m(\A)$, called the {\em $m$-periodic derived Hall algebra} of $\A$, is the $\mathbb{C}$-vector space with the basis $$\{u_{\bigoplus\limits_{i\in\mathbb{Z}_m}M_i[i]}~|~[M_i]\in {\rm Iso}(\A)~\text{for~all~}i\in\mathbb{Z}_m\}$$ and with the multiplication defined on basis elements by
{\begin{equation}\begin{split}&u_{\bigoplus\limits_{i\in\mathbb{Z}_m}A_i[i]}u_{\bigoplus\limits_{i\in\mathbb{Z}_m}B_i[i]}\\&=
v^{\zsum\lr{\sum\limits_{k=0}^{m-1}(-1)^k\hat{A}_{i+k},\hat{B}_i}}\sum\limits_{[I_i],[M_i]\in{\rm Iso}(\A), i\in\mathbb{Z}_m}
\prod\limits_{i\in\mathbb{Z}_m}\frac{H_{I_i[1]\oplus A_i,B_i\oplus I_{i-1}[-1]}^{M_i}}{a_{I_i}} u_{\bigoplus\limits_{i\in\mathbb{Z}_m}M_i[i]}.\end{split}\end{equation}}
\end{definition}
\begin{theorem}$($\cite{Zhang2}$)$\label{jimain}
The $m$-periodic derived Hall algebra $\mathcal {D}\mathcal {H}_m(\A)$ coincides with the dual derived Hall algebra of $D_m(\mathcal {A})$ defined by Xu-Chen \cite{XuChen}.
\end{theorem}

\section{An algebra embedding}
In this section, we give our main result as the following
\begin{theorem}\label{mainthm}
There exists an embedding of algebras $\varphi_m:\mathcal {D}\mathcal {H}_m(\A)\rightarrow \mathcal {D}\mathcal {H}_m^e(\A)$ defined by
$$u_{M_0}\mapsto u_{M_0}K_{-\frac{1}{2}\hat{M}_0},~\text{if}~m=1;$$
and
\begin{flalign*}u_{\bigoplus\limits_{i\in\mathbb{Z}_m}M_i[i]}\mapsto &v^{\frac{1}{4}\sum\limits_{i=1}^{m-1}\lr{\sum\limits_{k=0}^{m-1} (-1)^k{\hat{M}}_{i+k},\sum\limits_{k=0}^{m-1}(-1)^k\hat{M}_{i+1+k}}-\frac{1}{4}\lr{\sum\limits_{k=0}^{m-1}(-1)^k\hat{M}_{1+k},\sum\limits_{k=0}^{m-1}(-1)^k\hat{M}_{k}}
+\sum\limits_{i=0}^{m-1}\lr{\hat{M}_i,\sum\limits_{k=1}^{m-1}(-1)^k\hat{M}_{i+k}}}\\
&u_{\bigoplus\limits_{i\in\mathbb{Z}_m}M_i[i]}\prod\limits_{i\in\mathbb{Z}_m}K_{-\frac{1}{2}\sum\limits_{k=0}^{m-1}(-1)^k\hat{M}_{i+1+k},i},~\text{if}~m>1;\end{flalign*}
for any $M_i\in\A$ with $i\in\mathbb{Z}_m$.
\end{theorem}
\begin{proof}
We need to prove that
\begin{equation}\label{alghom}
\begin{split}&\varphi(u_{\bigoplus\limits_{i\in\mathbb{Z}_{m}}A_{i}[i]}u_{\bigoplus\limits_{i\in\mathbb{Z}_{m}}B_{i}[i]})=
\varphi(u_{\bigoplus\limits_{i\in\mathbb{Z}_{m}}A_{i}[i]})\varphi(u_{\bigoplus\limits_{i\in\mathbb{Z}_{m}}B_{i}[i]})
\end{split}\end{equation}
for any $A_i,B_i\in\A$ with $i\in\mathbb{Z}_m$.

If $m=1$, $\varphi(u_{A_0}u_{B_0})=v^{\lr{\hat{A}_0,\hat{B}_0}}\sum\limits_{[I_0],[M_0]\in{\rm Iso}(\A)}\frac{H_{I_0[1]\oplus A_0,B_0\oplus I_0[-1]}^{M_0}}{a_{I_0}}u_{M_0}K_{-\frac{1}{2}\hat{M}_0}$ and
\begin{flalign*}
\varphi(u_{A_0})\varphi(u_{B_0})&=u_{A_0}K_{-\frac{1}{2}\hat{A}_0}u_{B_0}K_{-\frac{1}{2}\hat{B}_0}=u_{A_0}u_{B_0}K_{-\frac{1}{2}(\hat{A}_0+\hat{B}_0)}\\
&=v^{\lr{\hat{A}_0,\hat{B}_0}}\sum\limits_{[I_0],[M_0]\in{\rm Iso}(\A)}\frac{H_{I_0[1]\oplus A_0,B_0\oplus I_0[-1]}^{M_0}}{a_{I_0}}u_{M_0}K_{\hat{I}_0}K_{-\frac{1}{2}(\hat{A}_0+\hat{B}_0)}\\
&=v^{\lr{\hat{A}_0,\hat{B}_0}}\sum\limits_{[I_0],[M_0]\in{\rm Iso}(\A)}\frac{H_{I_0[1]\oplus A_0,B_0\oplus I_0[-1]}^{M_0}}{a_{I_0}}u_{M_0}K_{\hat{I}_0-\frac{1}{2}(\hat{A}_0+\hat{B}_0)}.
\end{flalign*}
Since $\hat{M}_0=\hat{A}_0+\hat{B}_0-2\hat{I}_0$, we get that $\hat{I}_0-\frac{1}{2}(\hat{A}_0+\hat{B}_0)=-\frac{1}{2}\hat{M}_0$. Thus, we obtain that $\varphi(u_{A_0}u_{B_0})=\varphi(u_{A_0})\varphi(u_{B_0})$.

If $m>1$, on the one hand,
\begin{equation*}
\begin{split}&\varphi(u_{\bigoplus\limits_{i\in\mathbb{Z}_{m}}A_{i}[i]}u_{\bigoplus\limits_{i\in\mathbb{Z}_{m}}B_{i}[i]})=\\&
v^{\sum\limits_{i\in\mathbb{Z}_{m}}{\lr{\sum\limits_{k=0}^{m-1}(-1)^{k}
\hat{A}_{i+k},\hat{B}_{i}}}}\sum\limits_{[I_{i}],[M_{i}],i\in\mathbb{Z}_{m}}\prod\limits_{i\in\mathbb{Z}_{m}}\frac{H_{I_i[1]\oplus A_i,B_i\oplus I_{i-1}[-1]}^{M_i}}{a_{I_i}} v^{m_0}
u_{\bigoplus\limits_{i\in\mathbb{Z}_{m}}M_i[i]}\prod\limits_{i\in\mathbb{Z}_{m}}K_{-\frac{1}{2}\sum\limits_{k=0}^{m-1}(-1)^k\hat{M}_{i+1+k},i}
\end{split}\end{equation*}
where \begin{flalign*}m_0=&\frac{1}{4}\sum\limits_{i=1}^{m-1}\lr{\sum\limits_{k=0}^{m-1}(-1)^k{\hat{M}}_{i+k},\sum\limits_{k=0}^{m-1}(-1)^k\hat{M}_{i+1+k}}-
\frac{1}{4}\lr{\sum\limits_{k=0}^{m-1}(-1)^k\hat{M}_{1+k},\sum\limits_{k=0}^{m-1}(-1)^k\hat{M}_{k}}
\\&+\sum\limits_{i=0}^{m-1}\lr{\hat{M}_i,\sum\limits_{k=1}^{m-1}(-1)^k\hat{M}_{i+k}}.\end{flalign*}

On the other hand,
\begin{equation*}
\begin{split}&\varphi(u_{\bigoplus\limits_{i\in\mathbb{Z}_{m}}A_{i}[i]})\varphi(u_{\bigoplus\limits_{i\in\mathbb{Z}_{m}}B_{i}[i]})=\\&
v^{a_0}
u_{\bigoplus\limits_{i\in\mathbb{Z}_{m}}A_i[i]}\prod\limits_{i\in\mathbb{Z}_{m}}K_{-\frac{1}{2}\sum\limits_{k=0}^{m-1}(-1)^k\hat{A}_{i+1+k},i}
v^{b_0}
u_{\bigoplus\limits_{i\in\mathbb{Z}_{m}}B_i[i]}\prod\limits_{i\in\mathbb{Z}_{m}}K_{-\frac{1}{2}\sum\limits_{k=0}^{m-1}(-1)^k\hat{B}_{i+1+k},i}
\end{split}\end{equation*}
where
\begin{flalign*}a_0=&\frac{1}{4}\sum\limits_{i=1}^{m-1}\lr{\sum\limits_{k=0}^{m-1}(-1)^k{\hat{A}}_{i+k},\sum\limits_{k=0}^{m-1}(-1)^k\hat{A}_{i+1+k}}-
\frac{1}{4}\lr{\sum\limits_{k=0}^{m-1}(-1)^k\hat{A}_{1+k},\sum\limits_{k=0}^{m-1}(-1)^k\hat{A}_{k}}
\\&+\sum\limits_{i=0}^{m-1}\lr{\hat{A}_i,\sum\limits_{k=1}^{m-1}(-1)^k\hat{A}_{i+k}}\end{flalign*} and
\begin{flalign*}b_0=&\frac{1}{4}\sum\limits_{i=1}^{m-1}\lr{\sum\limits_{k=0}^{m-1}(-1)^k{\hat{B}}_{i+k},\sum\limits_{k=0}^{m-1}(-1)^k\hat{B}_{i+1+k}}-
\frac{1}{4}\lr{\sum\limits_{k=0}^{m-1}(-1)^k\hat{B}_{1+k},\sum\limits_{k=0}^{m-1}(-1)^k\hat{B}_{k}}
\\&+\sum\limits_{i=0}^{m-1}\lr{\hat{B}_i,\sum\limits_{k=1}^{m-1}(-1)^k\hat{B}_{i+k}}.\end{flalign*}
Note that
\begin{equation*}
\begin{split}&\prod\limits_{i\in\mathbb{Z}_{m}}K_{-\frac{1}{2}\sum\limits_{k=0}^{m-1}(-1)^k\hat{A}_{i+1+k},i}u_{\bigoplus\limits_{i\in\mathbb{Z}_{m}}B_i[i]}=
v^{t_0}u_{\bigoplus\limits_{i\in\mathbb{Z}_{m}}B_i[i]}\prod\limits_{i\in\mathbb{Z}_{m}}K_{-\frac{1}{2}\sum\limits_{k=0}^{m-1}(-1)^k\hat{A}_{i+1+k},i},
\\&\prod\limits_{i\in\mathbb{Z}_{m}}K_{-\frac{1}{2}\sum\limits_{k=0}^{m-1}(-1)^k\hat{A}_{i+1+k},i}
\prod\limits_{i\in\mathbb{Z}_{m}}K_{-\frac{1}{2}\sum\limits_{k=0}^{m-1}(-1)^k\hat{B}_{i+1+k},i}=
v^{t_1}
\prod\limits_{i\in\mathbb{Z}_{m}}K_{{-\frac{1}{2}\sum\limits_{k=0}^{m-1}(-1)^k(\hat{A}_{i+1+k}+\hat{B}_{i+1+k})},i},
\\&u_{\bigoplus\limits_{i\in\mathbb{Z}_m}A_i[i]}u_{\bigoplus\limits_{i\in\mathbb{Z}_m}B_i[i]}=
\sum\limits_{[I_i],[M_i]\in{\rm Iso}(\A), i\in\mathbb{Z}_m}v^{t_2}\prod\limits_{i\in\mathbb{Z}_m}\frac{H_{I_i[1]\oplus A_i,B_i\oplus I_{i-1}[-1]}^{M_i}}{a_{I_i}} u_{\bigoplus\limits_{i\in\mathbb{Z}_m}M_i[i]}\prod\limits_{i\in\mathbb{Z}_m}K_{{\hat{I}}_i,i},
\\&\prod\limits_{i\in\mathbb{Z}_{m}}K_{{\hat{I}}_i,i}\prod\limits_{i\in\mathbb{Z}_{m}}K_{-\frac{1}{2}\sum\limits_{k=0}^{m-1}(-1)^k(\hat{A}_{i+1+k}+\hat{B}_{i+1+k}),i}=
v^{t_3}
\prod\limits_{i\in\mathbb{Z}_{m}}K_{{\hat{I}_i-\frac{1}{2}\sum\limits_{k=0}^{m-1}(-1)^k(\hat{A}_{i+1+k}+\hat{B}_{i+1+k})},i},
\end{split}\end{equation*}
where \begin{flalign*}&t_0=\sum\limits_{i\in\mathbb{Z}_{m}}(-\frac{1}{2}\sum\limits_{k=0}^{m-1}(-1)^k\hat{A}_{i+1+k},\hat{B}_i-\hat{B}_{i+1}),\\&
t_1=-(-\frac{1}{2}\sum\limits_{k=0}^{m-1}(-1)^k\hat{A}_{k},-\frac{1}{2}\sum\limits_{k=0}^{m-1}(-1)^k\hat{B}_{1+k})+\sum\limits_{i=1}^{m-1}(-\frac{1}{2}\sum\limits_{k=0}^{m-1}(-1)^k\hat{A}_{i+1+k},-\frac{1}{2}\sum\limits_{k=0}^{m-1}(-1)^k\hat{B}_{i+k}),
\\&t_2=\sum\limits_{i\in\mathbb{Z}_m}\lr{\hat{A}_i,\hat{B}_i}+\sum\limits_{i\in\mathbb{Z}_m}\lr{\hat{M}_i-\hat{M}_{i+1},\hat{I}_i}+\sum\limits_{i=1}^{m-1}\lr{\hat{I}_{i-1},\hat{I}_i}-\lr{\hat{I}_0,\hat{I}_{m-1}},
\\&t_3=-(\hat{I}_{m-1},-\frac{1}{2}\sum\limits_{k=0}^{m-1}(-1)^k(\hat{A}_{1+k}+\hat{B}_{1+k}))+
\sum\limits_{i=1}^{m-1}(\hat{I}_i,-\frac{1}{2}\sum\limits_{k=0}^{m-1}(-1)^k(\hat{A}_{i+k}+\hat{B}_{i+k})).
\end{flalign*}

Using $\hat{A}_i+\hat{B}_i-\hat{I}_i-\hat{I}_{i-1}=\hat{M}_i$ for all $i\in\mathbb{Z}_m$, we can work out that
$$\sum\limits_{k=0}^{m-1}(-1)^k(\hat{A}_{i+1+k}+\hat{B}_{i+1+k})-2\hat{I}_i=\sum_{k=0}^{m-1}(-1)^k\hat{M}_{i+1+k}$$
for each $i\in\mathbb{Z}_m$.
Thus,
\begin{flalign*}
&\varphi(u_{\bigoplus\limits_{i\in\mathbb{Z}_{m}}A_{i}[i]})\varphi(u_{\bigoplus\limits_{i\in\mathbb{Z}_{m}}B_{i}[i]})=\\ &\sum\limits_{[I_{i}],[M_{i}],i\in\mathbb{Z}_{m}}v^{a_0+b_0+t_0+t_1+t_2+t_3}\prod\limits_{i\in\mathbb{Z}_{m}}\frac{H_{I_i[1]\oplus A_i,B_i\oplus I_{i-1}[-1]}^{M_i}}{a_{I_i}}
u_{\bigoplus\limits_{i\in\mathbb{Z}_{m}}M_i[i]}\prod\limits_{i\in\mathbb{Z}_{m}}K_{-\frac{1}{2}\sum\limits_{k=0}^{m-1}(-1)^k\hat{M}_{i+1+k},i}.
\end{flalign*}

In what follows, let us compare the exponents of $v$ on two sides of the equation (\ref{alghom}).

Using $\hat{A}_i+\hat{B}_i-\hat{I}_i-\hat{I}_{i-1}=\hat{M}_i$ for each $i\in\mathbb{Z}_m$,
we obtain that the exponent of $v$ on the left side of (\ref{alghom})
\begin{flalign*}&{\rm LHS}=\sum\limits_{i\in\mathbb{Z}_{m}}{\lr{\sum\limits_{k=0}^{m-1}(-1)^{k}\hat{A}_{i+k},\hat{B}_{i}}}+
\\&\frac{1}{4}\sum\limits_{i=1}^{m-1}
\lr{\sum\limits_{k=0}^{m-1}(-1)^k({\hat{A}}_{i+k}+{\hat{B}}_{i+k}-{\hat{I}}_{i+k}-{\hat{I}}_{i+k-1}),\sum\limits_{k=0}^{m-1}(-1)^k({\hat{A}}_{i+1+k}+{\hat{B}}_{i+1+k}-{\hat{I}}_{i+1+k}-{\hat{I}}_{i+k})}
\\&
-\frac{1}{4}\lr{\sum\limits_{k=0}^{m-1}(-1)^k({\hat{A}}_{1+k}+{\hat{B}}_{1+k}-{\hat{I}}_{1+k}-{\hat{I}}_{k}),\sum\limits_{k=0}^{m-1}(-1)^k({\hat{A}}_{k}+{\hat{B}}_{k}-{\hat{I}}_{k}-{\hat{I}}_{k-1})}
\\&
+\sum\limits_{i=0}^{m-1}\lr{{\hat{A}}_{i}+{\hat{B}}_{i}-{\hat{I}}_{i}-{\hat{I}}_{i-1},\sum\limits_{k=1}^{m-1}(-1)^k({\hat{A}}_{i+k}+{\hat{B}}_{i+k}-{\hat{I}}_{i+k}-{\hat{I}}_{i+k-1})}.
\end{flalign*}
Thus, we have that
\begin{flalign*}&{\rm LHS}=\sum\limits_{i\in\mathbb{Z}_{m}}{\lr{\sum\limits_{k=0}^{m-1}(-1)^{k}\hat{A}_{i+k},\hat{B}_{i}}}+
\frac{1}{4}\sum\limits_{i=1}^{m-1}\lr{\sum\limits_{k=0}^{m-1}(-1)^k{\hat{A}}_{i+k},\sum\limits_{k=0}^{m-1}(-1)^k\hat{A}_{i+1+k}}^{\bf(1)}\\&+
\frac{1}{4}\sum\limits_{i=1}^{m-1}\lr{\sum\limits_{k=0}^{m-1}(-1)^k{\hat{A}}_{i+k},\sum\limits_{k=0}^{m-1}(-1)^k\hat{B}_{i+1+k}}-
\frac{1}{4}\sum\limits_{i=1}^{m-1}\lr{\sum\limits_{k=0}^{m-1}(-1)^k{\hat{A}}_{i+k},\sum\limits_{k=0}^{m-1}(-1)^k({\hat{I}}_{i+1+k}+{\hat{I}}_{i+k})}
\\&+\frac{1}{4}\sum\limits_{i=1}^{m-1}\lr{\sum\limits_{k=0}^{m-1}(-1)^k{\hat{B}}_{i+k},\sum\limits_{k=0}^{m-1}(-1)^k\hat{A}_{i+1+k}}^{\bf(2)}+
\frac{1}{4}\sum\limits_{i=1}^{m-1}\lr{\sum\limits_{k=0}^{m-1}(-1)^k{\hat{B}}_{i+k},\sum\limits_{k=0}^{m-1}(-1)^k\hat{B}_{i+1+k}}^{\bf(3)}
\\&-\frac{1}{4}\sum\limits_{i=1}^{m-1}\lr{\sum\limits_{k=0}^{m-1}(-1)^k{\hat{B}}_{i+k},\sum\limits_{k=0}^{m-1}(-1)^k({\hat{I}}_{i+1+k}+{\hat{I}}_{i+k})}\\&
-\frac{1}{4}\sum\limits_{i=1}^{m-1}\lr{\sum\limits_{k=0}^{m-1}(-1)^k({\hat{I}}_{i+k}+{\hat{I}}_{i+k-1}),\sum\limits_{k=0}^{m-1}(-1)^k{\hat{A}}_{i+1+k}}\\
&-\frac{1}{4}\sum\limits_{i=1}^{m-1}\lr{\sum\limits_{k=0}^{m-1}(-1)^k({\hat{I}}_{i+k}+{\hat{I}}_{i+k-1}),\sum\limits_{k=0}^{m-1}(-1)^k{\hat{B}}_{i+k+1}}
\end{flalign*}
\begin{flalign*}
&+\frac{1}{4}\sum\limits_{i=1}^{m-1}\lr{\sum\limits_{k=0}^{m-1}(-1)^k({\hat{I}}_{i+k}+{\hat{I}}_{i+k-1}),\sum\limits_{k=0}^{m-1}(-1)^k({\hat{I}}_{i+1+k}+{\hat{I}}_{i+k})}\\&-
\frac{1}{4}\lr{\sum\limits_{k=0}^{m-1}(-1)^k{\hat{A}}_{1+k},\sum\limits_{k=0}^{m-1}(-1)^k\hat{A}_{k}}^{\bf(4)}-
\frac{1}{4}\lr{\sum\limits_{k=0}^{m-1}(-1)^k{\hat{A}}_{1+k},\sum\limits_{k=0}^{m-1}(-1)^k\hat{B}_{k}}
\\&+\frac{1}{4}\lr{\sum\limits_{k=0}^{m-1}(-1)^k{\hat{A}}_{1+k},\sum\limits_{k=0}^{m-1}(-1)^k({\hat{I}}_{k}+{\hat{I}}_{k-1})}
-\frac{1}{4}\lr{\sum\limits_{k=0}^{m-1}(-1)^k{\hat{B}}_{1+k},\sum\limits_{k=0}^{m-1}(-1)^k\hat{A}_{k}}^{\bf(5)}
\\&-\frac{1}{4}\lr{\sum\limits_{k=0}^{m-1}(-1)^k{\hat{B}}_{1+k},\sum\limits_{k=0}^{m-1}(-1)^k\hat{B}_{k}}^{\bf(6)}+
\frac{1}{4}\lr{\sum\limits_{k=0}^{m-1}(-1)^k{\hat{B}}_{1+k},\sum\limits_{k=0}^{m-1}(-1)^k({\hat{I}}_{k}+{\hat{I}}_{k-1})}\\
&+\frac{1}{4}\lr{\sum\limits_{k=0}^{m-1}(-1)^k({\hat{I}}_{1+k}+{\hat{I}}_{k}),\sum\limits_{k=0}^{m-1}(-1)^k{\hat{A}}_{k}}+
\frac{1}{4}\lr{\sum\limits_{k=0}^{m-1}(-1)^k({\hat{I}}_{1+k}+{\hat{I}}_{k}),\sum\limits_{k=0}^{m-1}(-1)^k{\hat{B}}_{k}}\\&-
\frac{1}{4}\lr{\sum\limits_{k=0}^{m-1}(-1)^k({\hat{I}}_{1+k}+{\hat{I}}_{k}),\sum\limits_{k=0}^{m-1}(-1)^k({\hat{I}}_{k}+{\hat{I}}_{k-1})}+
\sum\limits_{i=0}^{m-1}\lr{{\hat{A}}_{i},\sum\limits_{k=1}^{m-1}(-1)^k{\hat{A}}_{i+k}}^{\bf(7)}\\&+
\sum\limits_{i=0}^{m-1}\lr{{\hat{A}}_{i},\sum\limits_{k=1}^{m-1}(-1)^k{\hat{B}}_{i+k}}-
\sum\limits_{i=0}^{m-1}\lr{{\hat{A}}_{i},\sum\limits_{k=1}^{m-1}(-1)^k({\hat{I}}_{i+k}+{\hat{I}}_{i+k-1})}
\\&+\sum\limits_{i=0}^{m-1}\lr{{\hat{B}}_{i},\sum\limits_{k=1}^{m-1}(-1)^k{\hat{A}}_{i+k}}+
\sum\limits_{i=0}^{m-1}\lr{{\hat{B}}_{i},\sum\limits_{k=1}^{m-1}(-1)^k{\hat{B}}_{i+k}}^{\bf(8)}\\&-
\sum\limits_{i=0}^{m-1}\lr{{\hat{B}}_{i},\sum\limits_{k=1}^{m-1}(-1)^k({\hat{I}}_{i+k}+{\hat{I}}_{i+k-1})}-
\sum\limits_{i=0}^{m-1}\lr{{\hat{I}}_{i}+{\hat{I}}_{i-1},\sum\limits_{k=1}^{m-1}(-1)^k{\hat{A}}_{i+k}}\\&-
\sum\limits_{i=0}^{m-1}\lr{{\hat{I}}_{i}+{\hat{I}}_{i-1},\sum\limits_{k=1}^{m-1}(-1)^k{\hat{B}}_{i+k}}+
\sum\limits_{i=0}^{m-1}\lr{{\hat{I}}_{i}+{\hat{I}}_{i-1},\sum\limits_{k=1}^{m-1}(-1)^k({\hat{I}}_{i+k}+{\hat{I}}_{i+k-1})}.
\end{flalign*}

Using $\hat{A}_i+\hat{B}_i-\hat{I}_i-\hat{I}_{i-1}=\hat{M}_i$ for each $i\in\mathbb{Z}_m$,
we obtain that the exponent of $v$ on the right side of (\ref{alghom})
\begin{flalign*}&{\rm RHS}=\frac{1}{4}\sum\limits_{i=1}^{m-1}\lr{\sum\limits_{k=0}^{m-1}(-1)^k{\hat{A}}_{i+k},\sum\limits_{k=0}^{m-1}(-1)^k\hat{A}_{i+1+k}}^{\bf (1)}-
\frac{1}{4}\lr{\sum\limits_{k=0}^{m-1}(-1)^k\hat{A}_{1+k},\sum\limits_{k=0}^{m-1}(-1)^k\hat{A}_{k}}^{\bf (4)}
\\&+\sum\limits_{i=0}^{m-1}\lr{\hat{A}_i,\sum\limits_{k=1}^{m-1}(-1)^k\hat{A}_{i+k}}^{\bf (7)}+
\frac{1}{4}\sum\limits_{i=1}^{m-1}\lr{\sum\limits_{k=0}^{m-1}(-1)^k{\hat{B}}_{i+k},\sum\limits_{k=0}^{m-1}(-1)^k\hat{B}_{i+1+k}}^{\bf (3)}
\\&-\frac{1}{4}\lr{\sum\limits_{k=0}^{m-1}(-1)^k\hat{B}_{1+k},\sum\limits_{k=0}^{m-1}(-1)^k\hat{B}_{k}}^{\bf (6)}
+\sum\limits_{i=0}^{m-1}\lr{\hat{B}_i,\sum\limits_{k=1}^{m-1}(-1)^k\hat{B}_{i+k}}^{\bf (8)}\\&-\frac{1}{2}\sum\limits_{i\in\mathbb{Z}_{m}}{\lr{\sum\limits_{k=0}^{m-1}(-1)^k\hat{A}_{i+1+k},\hat{B}_i-\hat{B}_{i+1}}}-
\frac{1}{2}\sum\limits_{i\in\mathbb{Z}_{m}}{\lr{\hat{B}_i-\hat{B}_{i+1},\sum\limits_{k=0}^{m-1}(-1)^k\hat{A}_{i+1+k}}}
\end{flalign*}
\begin{flalign*}&-\frac{1}{4}{\lr{\sum\limits_{k=0}^{m-1}(-1)^k\hat{A}_{k},\sum\limits_{k=0}^{m-1}(-1)^k\hat{B}_{1+k}}}-
\frac{1}{4}\lr{\sum\limits_{k=0}^{m-1}(-1)^k\hat{B}_{1+k},\sum\limits_{k=0}^{m-1}(-1)^k\hat{A}_{k}}^{\bf (5)}
\\&+\frac{1}{4}\sum\limits_{i=1}^{m-1}{\lr{\sum\limits_{k=0}^{m-1}(-1)^k\hat{A}_{i+1+k},\sum\limits_{k=0}^{m-1}(-1)^k\hat{B}_{i+k}}}+
\frac{1}{4}\sum\limits_{i=1}^{m-1}\lr{\sum\limits_{k=0}^{m-1}(-1)^k\hat{B}_{i+k},\sum\limits_{k=0}^{m-1}(-1)^k\hat{A}_{i+1+k}}^{\bf (2)}
\\&+\sum\limits_{i\in\mathbb{Z}_m}\lr{\hat{A}_i,\hat{B}_i}+
\sum\limits_{i\in\mathbb{Z}_m}\lr{\hat{A}_i-\hat{A}_{i+1},\hat{I}_i}+\sum\limits_{i\in\mathbb{Z}_m}\lr{\hat{B}_i-\hat{B}_{i+1},\hat{I}_i}+
\sum\limits_{i\in\mathbb{Z}_m}\lr{\hat{I}_{i+1}-\hat{I}_{i-1},\hat{I}_i}
\\&+\sum\limits_{i=1}^{m-1}\lr{\hat{I}_{i-1},\hat{I}_i}^{\bf (11)}-\lr{\hat{I}_0,\hat{I}_{m-1}}^{\bf (14)}+
\frac{1}{2}\lr{\hat{I}_{m-1},\sum\limits_{k=0}^{m-1}(-1)^k\hat{A}_{1+k}}^{\bf (a)}+\frac{1}{2}\lr{\hat{I}_{m-1},\sum\limits_{k=0}^{m-1}(-1)^k\hat{B}_{1+k}}^{\bf (b)}\\
&+\frac{1}{2}\lr{{\sum\limits_{k=0}^{m-1}(-1)^k\hat{A}_{1+k},\hat{I}_{m-1}}}^{\bf (12)}+\frac{1}{2}\lr{\sum\limits_{k=0}^{m-1}(-1)^k\hat{B}_{1+k},\hat{I}_{m-1}}^{\bf (13)}-
\frac{1}{2}\sum\limits_{i=1}^{m-1}{\lr{\hat{I}_{i},\sum\limits_{k=0}^{m-1}(-1)^k\hat{A}_{i+k}}^{\bf (c)}}
\\&-\frac{1}{2}\sum\limits_{i=1}^{m-1}{\lr{\hat{I}_{i},\sum\limits_{k=0}^{m-1}(-1)^k\hat{B}_{i+k}}^{\bf (d)}}-
\frac{1}{2}\sum\limits_{i=1}^{m-1}\lr{{\sum\limits_{k=0}^{m-1}(-1)^k\hat{A}_{i+k},\hat{I}_{i}}}^{\bf (9)}-\frac{1}{2}\sum\limits_{i=1}^{m-1}\lr{\sum\limits_{k=0}^{m-1}(-1)^k\hat{B}_{i+k},\hat{I}_{i}}^{\bf (10)}.
\end{flalign*}

It is easy to see that \begin{equation}\label{tihuan}\sum\limits_{k=0}^{m-1}(-1)^k({\hat{I}}_{i+k}+{\hat{I}}_{i+k-1})=2\hat{I}_{i-1}~\text{and}~ \sum\limits_{k=1}^{m-1}(-1)^k({\hat{I}}_{i+k}+{\hat{I}}_{i+k-1})=\hat{I}_{i-1}-\hat{I}_{i}\end{equation} for each $i\in\mathbb{Z}_{m}$.

Substituting (\ref{tihuan}) into LHS and eliminating the common terms {({\bf 1})}-{({\bf 8})} with RHS, which have been marked with numbers in bold, we obtain that the remaining terms in LHS

\begin{flalign*}&{\rm LHS'} = \sum\limits_{i \in \mathbb{Z}_{m}}\lr{\sum\limits_{k=0}\limits^{m-1} (-1)^k\hat{A}_{i+k},\hat{B}_{i}}+\frac{1}{4}\sum\limits_{i=1}\limits^{m-1} \lr{\sum\limits_{k=0}\limits^{m-1}(-1)^{k} \hat{A}_{i+k},\sum\limits_{k=0}\limits^{m-1}(-1)^{k}\hat{B}_{i+k+1}}\\&-\frac{1}{2}\sum\limits_{i=1}\limits^{m-1} \lr{\sum\limits_{k=0}\limits^{m-1}(-1)^{k} \hat{A}_{i+k},\hat{I}_{i}}^{\bf (9)}-\frac{1}{2}\sum\limits_{i=1}\limits^{m-1} \lr{\sum\limits_{k=0}\limits^{m-1}(-1)^{k} \hat{B}_{i+k},\hat{I}_{i}}^{\bf (10)}\\&-\frac{1}{2}\sum\limits_{i=1}\limits^{m-1} \lr{\hat{I}_{i-1},\sum\limits_{k=0}\limits^{m-1}(-1)^{k}\hat{A}_{i+1+k}}^{\bf (a')}-\frac{1}{2}\sum\limits_{i=1}\limits^{m-1} \lr{ \hat{I}_{i-1},\sum\limits_{k=0}\limits^{m-1}(-1)^{k}\hat{B}_{i+1+k}}^{\bf (b')} \\&+\sum\limits_{i=1}\limits^{m-1} \lr{ \hat{I}_{i-1},\hat{I}_{i}}^{\bf (11)}-\frac{1}{4} \lr{\sum\limits_{k=0}\limits^{m-1}(-1)^{k} \hat{A}_{1+k},\sum\limits_{k=0}\limits^{m-1}(-1)^{k}\hat{B}_{k}}+\frac{1}{2} \lr{\sum\limits_{k=0}\limits^{m-1}(-1)^{k} \hat{A}_{1+k},\hat{I}_{m-1}}^{\bf (12)}\\&+\frac{1}{2}\lr{\sum\limits_{k=0}\limits^{m-1}(-1)^{k} \hat{B}_{1+k},\hat{I}_{m-1}}^{\bf (13)}+\frac{1}{2}\lr{\hat{I}_{0},\sum\limits_{k=0}\limits^{m-1}(-1)^{k}\hat{A}_{k}}^{\bf (c')}+\frac{1}{2}\lr{\hat{I}_{0},\sum\limits_{k=0}\limits^{m-1}(-1)^{k}\hat{B}_{k}}^{\bf (d')}-\lr{\hat{I_{0}},\hat{I}_{m-1}}^{\bf (14)}\\
&+\sum\limits_{i=0}\limits^{m-1} \lr{\hat{A}_{i},\sum\limits_{k=1}\limits^{m-1}(-1)^{k}\hat{B}_{i+k}}-\sum\limits_{i=0}\limits^{m-1} \lr{ \hat{A}_{i},\hat{I}_{i-1}-\hat{I}_{i}}+\sum\limits_{i=0}\limits^{m-1} \lr{\hat{B}_{i},\sum\limits_{k=1}\limits^{m-1}(-1)^{k}\hat{A}_{i+k}}-\sum\limits_{i=0}\limits^{m-1} \lr{ \hat{B}_{i},\hat{I}_{i-1}-\hat{I}_{i}}\end{flalign*}
\begin{flalign*}
&-\sum\limits_{i=0}\limits^{m-1} \lr{\hat{I}_{i}+\hat{I}_{i-1},\sum\limits_{k=1}\limits^{m-1}(-1)^{k}\hat{A}_{i+k}}-\sum\limits_{i=0}\limits^{m-1} \lr{ \hat{I}_{i}+\hat{I}_{i-1},\sum\limits_{k=1}\limits^{m-1}(-1)^{k}\hat{B}_{i+k}}+\sum\limits_{i=0}\limits^{m-1} \lr{\hat{I}_{i}+\hat{I}_{i-1},\hat{I}_{i-1}-\hat{I}_{i}}.
\end{flalign*}

Eliminating the common terms {({\bf 9})}-{({\bf 14})}  with RHS and combining the terms $({\bf a}),({\bf b}),({\bf c}),({\bf d})$ in RHS with $({\bf a}'),({\bf b}'),({\bf c}'),({\bf d}')$ in LHS$'$, respectively, we obtain that

\begin{flalign*}&{\rm RHS}-{\rm LHS}=\frac{1}{2}\sum\limits_{i\in \mathbb{Z}_{m}} \lr{\sum\limits_{k=0}\limits^{m-1}(-1)^{k} \hat{A}_{i+k},\hat{B}_{i}}^{(\bf i)}-\frac{1}{2}\sum\limits_{i\in \mathbb{Z}_{m}} \lr{\sum\limits_{k=0}\limits^{m-1}(-1)^{k} \hat{A}_{i+k+1},\hat{B}_{i}}\\&
+\frac{1}{2}\sum\limits_{i\in \mathbb{Z}_{m}} \lr{\hat{B}_{i},\sum\limits_{k=0}\limits^{m-1}(-1)^{k}\hat{A}_{i+k}}-\frac{1}{2}\sum\limits_{i\in \mathbb{Z}_{m}} \lr{\hat{B}_{i},\sum\limits_{k=0}\limits^{m-1}(-1)^{k}\hat{A}_{i+k+1}}\\&-\frac{1}{4} \lr{\sum\limits_{k=0}\limits^{m-1}(-1)^{k} \hat{A}_{k},\sum\limits_{k=0}\limits^{m-1}(-1)^{k}\hat{B}_{k+1}}^{({\bf e})}
+\frac{1}{4}\sum\limits_{i=1}\limits^{m-1} \lr{\sum\limits_{k=0}\limits^{m-1}(-1)^{k} \hat{A}_{i+k+1},\sum\limits_{k=0}\limits^{m-1}(-1)^{k}\hat{B}_{i+k}}^{(\bf f)}\\&+\sum\limits_{i\in \mathbb{Z}_{m}} \lr{\hat{A}_{i},\hat{B}_{i}}+\sum\limits_{i\in \mathbb{Z}_{m}} \lr{ \hat{A}_{i},\hat{I}_{i}}-\sum\limits_{i\in \mathbb{Z}_{m}} \lr{ \hat{A}_{i},\hat{I}_{i-1}}
+\sum\limits_{i\in \mathbb{Z}_{m}} \lr{ \hat{B}_{i},\hat{I}_{i}}-\sum\limits_{i\in \mathbb{Z}_{m}} \lr{ \hat{B}_{i},\hat{I}_{i-1}}\\&+\sum\limits_{i\in \mathbb{Z}_{m}} \lr{\hat{I}_{i},\hat{I}_{i-1}}-\sum\limits_{i\in \mathbb{Z}_{m}} \lr{\hat{I}_{i-1},\hat{I}_{i}}+\frac{1}{2}\sum\limits_{i\in \mathbb{Z}_{m}} \lr{ \hat{I}_{i-1},\sum\limits_{k=0}\limits^{m-1}(-1)^{k}\hat{A}_{i+k+1}}\\&+\frac{1}{2}\sum\limits_{i\in \mathbb{Z}_{m}} \lr{\hat{I}_{i-1},\sum\limits_{k=0}\limits^{m-1}(-1)^{k}\hat{B}_{i+k+1}}-\frac{1}{2}\sum\limits_{i\in \mathbb{Z}_{m}} \lr{\hat{I}_{i},\sum\limits_{k=0}\limits^{m-1}(-1)^{k}\hat{A}_{i+k}}\\&-\frac{1}{2}\sum\limits_{i\in \mathbb{Z}_{m}} \lr{\hat{I}_{i},\sum\limits_{k=0}\limits^{m-1}(-1)^{k}\hat{B}_{i+k}}-\sum\limits_{i\in \mathbb{Z}_{m}} \lr{\sum\limits_{k=0}\limits^{m-1}(-1)^{k} \hat{A}_{i+k},\hat{B}_{i}}^{(\bf i')}\\&-\frac{1}{4}\sum\limits_{i=1}\limits^{m-1} \lr{\sum\limits_{k=0}\limits^{m-1}(-1)^{k} \hat{A}_{i+k},\sum\limits_{k=0}\limits^{m-1}(-1)^{k}\hat{B}_{i+k+1}}^{(\bf e')}+\frac{1}{4} \lr{\sum\limits_{k=0}\limits^{m-1}(-1)^{k} \hat{A}_{k+1},\sum\limits_{k=0}\limits^{m-1}(-1)^{k}\hat{B}_{k}}^{(\bf f')}\\&-\sum\limits_{i\in \mathbb{Z}_{m}} \lr{\hat{A}_{i},\sum\limits_{k=1}\limits^{m-1}(-1)^{k}\hat{B}_{i+k}}+\sum\limits_{i\in \mathbb{Z}_{m}} \lr{\hat{A}_{i},\hat{I}_{i-1}-\hat{I}_{i}}-\sum\limits_{i\in \mathbb{Z}_{m}} \lr{\hat{B}_{i},\sum\limits_{k=1}\limits^{m-1}(-1)^{k}\hat{A}_{i+k}}\\&+\sum\limits_{i\in \mathbb{Z}_{m}} \lr{\hat{B}_{i},\hat{I}_{i-1}-\hat{I}_{i}}+\sum\limits_{i\in \mathbb{Z}_{m}} \lr{\hat{I}_{i},\sum\limits_{k=1}\limits^{m-1}(-1)^{k}\hat{A}_{i+k}}^{(\bf g)}+\sum\limits_{i\in  \mathbb{Z}_{m}} \lr{\hat{I}_{i},\sum\limits_{k=1}\limits^{m-1}(-1)^{k}\hat{A}_{i+k+1}}^{(\bf g')}\\&+\sum\limits_{i\in \mathbb{Z}_{m}} \lr{\hat{I}_{i},\sum\limits_{k=1}\limits^{m-1}(-1)^{k}\hat{B}_{i+k}}^{(\bf h)}+\sum\limits_{i\in \mathbb{Z}_{m}}\lr{\hat{I}_{i},\sum\limits_{k=1}\limits^{m-1}(-1)^{k}\hat{B}_{i+k+1}}^{(\bf h')}-\sum\limits_{i\in \mathbb{Z}_{m}} \lr{\hat{I}_{i}+\hat{I}_{i-1},\hat{I}_{i-1}-\hat{I}_{i}}.
\end{flalign*}
Let us combine the terms $({\bf e}),({\bf f}),({\bf g}),({\bf h}),({\bf i})$ in RHS--LHS with $({\bf e}'),({\bf f}'),({\bf g}'),({\bf h}'),({\bf i}')$, respectively, we have that
\begin{flalign*}
({\bf e})+({\bf e}')=-\frac{1}{4}\sum\limits_{i\in\mathbb{Z}_m}\lr{\sum\limits_{k\in\mathbb{Z}_m}(-1)^k\hat{A}_{i+k},\sum\limits_{k\in\mathbb{Z}_m}(-1)^k\hat{B}_{i+k+1}},\end{flalign*}
\begin{flalign*}({\bf f})+({\bf f}')=\frac{1}{4}\sum\limits_{i\in\mathbb{Z}_m}\lr{\sum\limits_{k\in\mathbb{Z}_m}(-1)^k\hat{A}_{i+k+1},\sum\limits_{k\in\mathbb{Z}_m}(-1)^k\hat{B}_{i+k}},
\end{flalign*}
\begin{flalign*}
({\bf g})+({\bf g}')&=\sum\limits_{i\in \mathbb{Z}_{m}} \lr{\hat{I}_{i},\sum\limits_{k=0}^{m-1}(-1)^{k}\hat{A}_{i+k}}-\sum\limits_{i\in \mathbb{Z}_{m}} \lr{\hat{I}_{i},\hat{A}_{i}}\\&+\sum\limits_{i\in \mathbb{Z}_{m}} \lr{\hat{I}_{i},\sum\limits_{k=0}^{m-1}(-1)^{k}\hat{A}_{i+k+1}}-\sum\limits_{i\in \mathbb{Z}_{m}} \lr{\hat{I}_{i},\hat{A}_{i+1}}\\&=\sum\limits_{i\in \mathbb{Z}_{m}} (\lr{\hat{I}_{i},2\hat{A}_{i}}-\lr{\hat{I}_{i},\hat{A}_{i}}-\lr{\hat{I}_{i},\hat{A}_{i+1}})\\&=\sum\limits_{i\in \mathbb{Z}_{m}}(\lr{\hat{I}_{i},\hat{A}_{i}}-\lr{\hat{I}_{i},\hat{A}_{i+1}}),
\end{flalign*}
$$({\bf h})+({\bf h}')=\sum\limits_{i\in \mathbb{Z}_{m}}(\lr{\hat{I}_{i},\hat{B}_{i}}-\lr{\hat{I}_{i},\hat{B}_{i+1}}),$$ and
$$({\bf i})+({\bf i}')=-\frac{1}{2}\sum\limits_{i\in \mathbb{Z}_{m}} \lr{\sum\limits_{k\in\mathbb{Z}_m}(-1)^{k} \hat{A}_{i+k},\hat{B}_{i}}.$$
Thus, we obtain that
\begin{flalign*}&{\rm RHS}-{\rm LHS}=-\frac{1}{2}\sum\limits_{i\in \mathbb{Z}_{m}} \lr{\sum\limits_{k\in \mathbb{Z}_{m}}(-1)^{k} \hat{A}_{i+k+1},\hat{B}_{i}}^{(\bf m)}+\frac{1}{2}\sum\limits_{i\in \mathbb{Z}_{m}}\lr{\hat{B}_{i},\sum\limits_{k\in \mathbb{Z}_{m}}(-1)^{k}\hat{A}_{i+k}}^{(\bf j)}\\&-\frac{1}{2}\sum\limits_{i\in \mathbb{Z}_{m}}\lr{\hat{B}_{i},\sum\limits_{k\in \mathbb{Z}_{m}}(-1)^{k}\hat{A}_{i+k+1}}^{(\bf j'')}-\frac{1}{4}\sum\limits_{i\in \mathbb{Z}_{m}} \lr{\sum\limits_{k\in \mathbb{Z}_{m}}(-1)^{k} \hat{A}_{i+k},\sum\limits_{k\in \mathbb{Z}_{m}}(-1)^{k}\hat{B}_{i+k+1}}^{(\bf k)}\\&+\frac{1}{4}\sum\limits_{i\in \mathbb{Z}_{m}}\lr{\sum\limits_{k\in \mathbb{Z}_{m}}(-1)^{k} \hat{A}_{i+k+1},\sum\limits_{k\in \mathbb{Z}_{m}}(-1)^{k}\hat{B}_{i+k}}^{(\bf k')}+\sum\limits_{i\in \mathbb{Z}_{m}}\lr{\hat{A}_{i},\hat{B}_{i}}+\sum\limits_{i\in \mathbb{Z}_{m}}\lr{\hat{A}_{i},\hat{I}_{i}}\\&-\sum\limits_{i\in \mathbb{Z}_{m}}\lr{\hat{A}_{i},\hat{I}_{i-1}}+\sum\limits_{i\in \mathbb{Z}_{m}} \lr{\hat{B}_{i},\hat{I}_{i}}-\sum\limits_{i\in  \mathbb{Z}_{m}}\lr{\hat{B}_{i},\hat{I}_{i-1}}+\sum\limits_{i\in \mathbb{Z}_{m}}\lr{\hat{I}_{i},\hat{I}_{i-1}}-\sum\limits_{i\in \mathbb{Z}_{m}}\lr{\hat{I}_{i-1},\hat{I}_{i}}\\&+\frac{1}{2}\sum\limits_{i\in \mathbb{Z}_{m}}\lr{ \hat{I}_{i-1},\sum\limits_{k\in  \mathbb{Z}_{m}}(-1)^{k}\hat{A}_{i+k+1}}^{(\bf l)}+\frac{1}{2}\sum\limits_{i\in \mathbb{Z}_{m}}\lr{ \hat{I}_{i-1},\sum\limits_{k\in \mathbb{Z}_{m}}(-1)^{k}\hat{B}_{i+k+1}}^{(\bf n)}\\&-\frac{1}{2}\sum\limits_{i\in \mathbb{Z}_{m}}\lr{\hat{I}_{i},\sum\limits_{k\in \mathbb{Z}_{m}}(-1)^{k}\hat{A}_{i+k}}^{(\bf l')}-\frac{1}{2}\sum\limits_{i\in \mathbb{Z}_{m}}\lr{\hat{I}_{i},\sum\limits_{k\in \mathbb{Z}_{m}}(-1)^{k}\hat{B}_{i+k}}^{(\bf n')}\\&-\frac{1}{2}\sum\limits_{i\in \mathbb{Z}_{m}}\lr{ \sum\limits_{k\in \mathbb{Z}_{m}}(-1)^{k}\hat{A}_{i+k},\hat{B}_{i}}^{(\bf m')}-\sum\limits_{i\in \mathbb{Z}_{m}}\lr{\hat{A}_{i},\sum\limits_{k\in \mathbb{Z}_{m}}(-1)^{k}\hat{B}_{i+k}}+\sum\limits_{i\in  \mathbb{Z}_{m}}\lr{ \hat{A}_{i},\hat{B}_{i}}\\&+\sum\limits_{i\in \mathbb{Z}_{m}}\lr{ \hat{A}_{i},\hat{I}_{i-1}-\hat{I}_{i}}-\sum\limits_{i\in \mathbb{Z}_{m}}\lr{ \hat{B}_{i},\sum\limits_{k\in \mathbb{Z}_{m}}(-1)^{k}\hat{A}_{i+k}}^{(\bf j')}+\sum\limits_{i\in \mathbb{Z}_{m}}\lr{\hat{B}_{i},\hat{A}_{i}}+\sum\limits_{i\in \mathbb{Z}_{m}}\lr{\hat{B}_{i},\hat{I}_{i-1}-\hat{I}_{i}}\\&+\sum\limits_{i\in \mathbb{Z}_{m}}\lr{ \hat{I}_{i},\hat{A}_{i}}-\sum\limits_{i\in \mathbb{Z}_{m}}\lr{ \hat{I}_{i-1},\hat{A}_{i}}+\sum\limits_{i\in \mathbb{Z}_{m}}\lr{ \hat{I}_{i},\hat{B}_{i}}-\sum\limits_{i\in \mathbb{Z}_{m}}\lr{ \hat{I}_{i-1},\hat{B}_{i}}+\sum\limits_{i\in \mathbb{Z}_{m}}\lr{ \hat{I}_{i},\hat{I}_{i}}\\&-\sum\limits_{i\in \mathbb{Z}_{m}}\lr{ \hat{I}_{i},\hat{I}_{i-1}}+\sum\limits_{i\in \mathbb{Z}_{m}}\lr{ \hat{I}_{i-1},\hat{I}_{i}}-\sum\limits_{i\in  \mathbb{Z}_{m}}\lr{ \hat{I}_{i-1},\hat{I}_{i-1}}.
\end{flalign*}
Note that \begin{flalign*}
({\bf j})+({\bf j}')+({\bf j}'')&=-\frac{1}{2}\sum\limits_{i\in\mathbb{Z}_m}\lr{\hat{B}_i,\sum\limits_{k\in\mathbb{Z}_m}(-1)^k(\hat{A}_{i+k}+\hat{A}_{i+k+1})}
\\&=-\frac{1}{2}\sum\limits_{i\in\mathbb{Z}_m}\lr{\hat{B}_i,2\hat{A}_i}=-\sum\limits_{i\in\mathbb{Z}_m}\lr{\hat{B}_i,\hat{A}_i}
\end{flalign*}
and
\begin{flalign*}
({\bf k})+({\bf k}')&=\frac{1}{4}\sum\limits_{i\in\mathbb{Z}_m}\lr{\sum\limits_{k\in\mathbb{Z}_m}(-1)^k(\hat{A}_{i+k+1}-\hat{A}_{i+k-1}),\sum\limits_{k\in\mathbb{Z}_m}(-1)^k\hat{B}_{i+k}}
\\&=\frac{1}{4}\sum\limits_{i\in\mathbb{Z}_m}\lr{2(\hat{A}_i-\hat{A}_{i-1}),\sum\limits_{k\in\mathbb{Z}_m}(-1)^k\hat{B}_{i+k}}
\\&=\frac{1}{2}\sum\limits_{i\in\mathbb{Z}_m}\lr{\hat{A}_i-\hat{A}_{i-1},\sum\limits_{k\in\mathbb{Z}_m}(-1)^k\hat{B}_{i+k}}.
\end{flalign*}
Similarly, we have that
$({\bf l})+({\bf l}')=\sum\limits_{i\in\mathbb{Z}_m}\lr{\hat{I}_{i-1},\hat{A}_i-\hat{A}_{i-1}}$,
$({\bf m})+({\bf m}')=-\sum\limits_{i\in\mathbb{Z}_m}\lr{\hat{A}_i,\hat{B}_i}$ and
$({\bf n})+({\bf n}')=\sum\limits_{i\in\mathbb{Z}_m}\lr{\hat{I}_{i-1},\hat{B}_i-\hat{B}_{i-1}}$.
Hence,
\begin{flalign*}&{\rm RHS}-{\rm LHS}=-\sum\limits_{i\in \mathbb{Z}_{m}}\lr{\hat{B}_{i},\hat{A}_{i}}+\frac{1}{2}\sum\limits_{i\in \mathbb{Z}_{m}}\lr{\hat{A}_{i},\sum\limits_{k\in \mathbb{Z}_{m}}(-1)^{k}\hat{B}_{i+k}}^{(\bf x)}-\frac{1}{2}\sum\limits_{i\in \mathbb{Z}_{m}}\lr{\hat{A}_{i},\sum\limits_{k\in \mathbb{Z}_{m}}(-1)^{k}\hat{B}_{i+k+1}}^{(\bf y)}\\&+\sum\limits_{i\in \mathbb{Z}_{m}}\lr{\hat{I}_{i-1},\hat{A}_{i}}-\sum\limits_{i\in \mathbb{Z}_{m}} \lr{\hat{I}_{i},\hat{A}_{i}}-\sum\limits_{i\in \mathbb{Z}_{m}} \lr{\hat{A}_{i},\hat{B}_{i}}+\sum\limits_{i\in \mathbb{Z}_{m}} \lr{\hat{I}_{i-1},\hat{B}_{i}}-\sum\limits_{i\in \mathbb{Z}_{m}} \lr{\hat{I}_{i},\hat{B}_{i}}
+\sum\limits_{i\in \mathbb{Z}_{m}}\lr{\hat{A}_{i},\hat{B}_{i}}\\&+\sum\limits_{i\in \mathbb{Z}_{m}}\lr{\hat{A}_{i},\hat{I}_{i}}-\sum\limits_{i\in \mathbb{Z}_{m}}\lr{\hat{A}_{i},\hat{I}_{i-1}}+\sum\limits_{i\in \mathbb{Z}_{m}} \lr{\hat{B}_{i},\hat{I}_{i}}-\sum\limits_{i\in  \mathbb{Z}_{m}}\lr{\hat{B}_{i},\hat{I}_{i-1}}+\sum\limits_{i\in \mathbb{Z}_{m}}\lr{\hat{I}_{i},\hat{I}_{i-1}}-\sum\limits_{i\in \mathbb{Z}_{m}}\lr{\hat{I}_{i-1},\hat{I}_{i}}
\\&-\sum\limits_{i\in \mathbb{Z}_{m}}\lr{\hat{A}_{i},\sum\limits_{k\in \mathbb{Z}_{m}}(-1)^{k}\hat{B}_{i+k}}^{(\bf z)}+\sum\limits_{i\in  \mathbb{Z}_{m}}\lr{ \hat{A}_{i},\hat{B}_{i}}+\sum\limits_{i\in \mathbb{Z}_{m}}\lr{ \hat{A}_{i},\hat{I}_{i-1}-\hat{I}_{i}}+\sum\limits_{i\in \mathbb{Z}_{m}}\lr{\hat{B}_{i},\hat{A}_{i}}\\&+\sum\limits_{i\in \mathbb{Z}_{m}}\lr{\hat{B}_{i},\hat{I}_{i-1}-\hat{I}_{i}}+\sum\limits_{i\in \mathbb{Z}_{m}}\lr{ \hat{I}_{i},\hat{A}_{i}}-\sum\limits_{i\in \mathbb{Z}_{m}}\lr{ \hat{I}_{i-1},\hat{A}_{i}}+\sum\limits_{i\in \mathbb{Z}_{m}}\lr{ \hat{I}_{i},\hat{B}_{i}}-\sum\limits_{i\in \mathbb{Z}_{m}}\lr{ \hat{I}_{i-1},\hat{B}_{i}}\\&+\sum\limits_{i\in \mathbb{Z}_{m}}\lr{ \hat{I}_{i},\hat{I}_{i}}-\sum\limits_{i\in \mathbb{Z}_{m}}\lr{ \hat{I}_{i},\hat{I}_{i-1}}+\sum\limits_{i\in \mathbb{Z}_{m}}\lr{ \hat{I}_{i-1},\hat{I}_{i}}-\sum\limits_{i\in  \mathbb{Z}_{m}}\lr{ \hat{I}_{i-1},\hat{I}_{i-1}}.
\end{flalign*}
Noting that
$({\bf x})+({\bf y})+({\bf z})=-\sum\limits_{i\in\mathbb{Z}_m}\lr{\hat{A}_i,\hat{B}_i}$, we can easily obtain that ${\rm RHS}-{\rm LHS}=0$, i.e., ${\rm RHS}={\rm LHS}$.
Therefore, we complete the proof.
\end{proof}

\begin{remark}
$(1)$~The $m=1$ case of Theorem \ref{mainthm} has been similarly considered in \cite[Remark 5.6]{CLR}. According to \cite[Definition 5.3]{LinP}, we have the extended semi-derived Ringel--Hall algebra of $\A$. By \cite[Theorem 5.2]{Zhang2}, the $m$-periodic extended derived Hall algebra $\mathcal {D}\mathcal {H}_m^{\rm e}(\A)$ is isomorphic to a twisted version $\mathcal {SD}\mathcal {H}_{\mathbb{Z}_m,\frac{1}{2}}(\A)$ of the extended semi-derived Ringel--Hall algebra, where the twisting is in the sense of Bridgeland (cf. \cite{ZHC2}). Hence, Theorem \ref{mainthm} can give an embedding of algebras from the odd periodic derived Hall algebra $\mathcal {D}\mathcal {H}_m(\A)$ to $\mathcal {SD}\mathcal {H}_{\mathbb{Z}_m,\frac{1}{2}}(\A)$. This is essentially the algebra embedding given in \cite{LinP}, but the extended semi-derived Ringel--Hall algebra in \cite{LinP} is the untwisted version.

$(2)$~By \cite{Zhang2}, Definition \ref{maindef} also applies to even periodic derived category $D_m(\mathcal {A})$, but Definition \ref{jidef} does not. That is, the derived Hall algebra (without adding $K$-elements) of $D_m(\mathcal {A})$ for even $m$ has not been well-defined (cf. \cite[Remark 4.3]{Zhang2}). Hence, there does not exist a version of Theorem \ref{mainthm} for even $m$.
\end{remark}

\section*{Acknowledgments}
The authors would like to thank the anonymous referee for the helpful comments and suggestions. This work was partially supported by the National Natural Science Foundation of China (No. 12271257).

\end{document}